\documentclass[12pt]{amsart}
\usepackage{amssymb}
\usepackage{amsmath}
\usepackage[active]{srcltx}
\usepackage{t1enc}
\usepackage[latin2]{inputenc}
\usepackage{verbatim}
\usepackage{amsmath,amsfonts,amssymb,amsthm}
\usepackage[mathcal]{eucal}
\usepackage{enumerate}
\usepackage[centertags]{amsmath}
\usepackage{graphics}

\newtheorem{theorem}{Theorem}

\newtheorem{lemma}{Lemma}

\newtheorem{corollary}{Corollary}

\begin{document}
\author{Davit Baramidze, Zura dvalashvili and Giorgi Tutberidze}
\title[ N\"orlund means]{Convergence of N\"orlund means with respect to Vilenkin systems of integrable functions}
\address{D. Baramidze, The University of Georgia, School of science and
	technology, 77a Merab Kostava St, Tbilisi 0128, Georgia and Department of
	Computer Science and Computational Engineering, UiT - The Arctic University
	of Norway, P.O. Box 385, N-8505, Narvik, Norway.}
\email{davit.baramidze@ug.edu.ge }
\address{Z. dvalashvili, The University of Georgia, School of science and
	technology, 77a Merab Kostava St, Tbilisi 0128, Georgia.}
\email{zurabdvalashvili@gmail.com}
\address{G. Tutberidze, The University of Georgia, School of science and
	technology, 77a Merab Kostava St, Tbilisi 0128, Georgia.}
\email{g.tutberidze@ug.edu.ge }
\thanks{The research was supported by Shota Rustaveli National Science
	Foundation grant no. FR-19-676.}
\date{}
\maketitle

\begin{abstract}
In this paper we derive converge of N\"orlund means of Vilenkin-Fourier
series with monotone coefficients of integrable functions in Lebesgue and
Vilinkin-Lebesgue points. Moreover, we discuss pointwise convergence and convergence in $L_p$ norms of such Nörlund means.
\end{abstract}



\section{Introduction}

\bigskip\ 

It is well-known (see e.g. the books \cite{PTWeisz} and \cite{sws}) that there exists an
absolute constant $c_{p},$ depending only on $p,$ such that 
\begin{equation*}
\left\Vert S_{n}f\right\Vert _{p}\leq c_{p}\left\Vert f\right\Vert _{p},%
\text{ \ when \ }p>1.
\end{equation*}%
On the other hand, (for details see \cite{BPT1,PTW3,tep6,tep7,tep8,tep9,Tut1}) boundedness does not hold for $p=1.$ The analogue of Carleson's theorem for Walsh system was proved by Billard 
\cite{Billard1967} for $p=2$ and by Sjölin \cite{sj1} for $1 <p<\infty$,
while for bounded Vilenkin systems by Gosselin \cite{goles}. For
Walsh-Fourier series, Schipp \cite{sws} gave a proof by using methods
of martingale theory. A similar proof for Vilenkin-Fourier series can be
found in Schipp and Weisz \cite{s3,wk} (see also \cite{PSTW} and \cite{PTWeisz}). In each proof, they show that 
maximal operator of partial sums is bounded on $L_p$, i.e., there exists
an absolute constant $c_p$ such that 
\begin{equation*}
\left\Vert S^{\ast }f\right\Vert_p\leq c_p\left\Vert f\right\Vert_p,\text{ \
	when \ }f\in L_p,\text{ \ } p>1.
\end{equation*}
Hence, if $f\in L_p(G_m),$ where $p>1,$ then 
$
S_{n}f\to f, \ \ \text{a.e. on } \ \ G_m.
$
Stein \cite{steindiv} constructed the integrable function whose
Vilenkin-Fourier (Walsh-Fourier) series diverges almost everywhere. In \cite%
{sws} was proved that there exists an integrable function whose
Walsh-Fourier series diverges everywhere. The a.e convergence of subsequences of Vilenkin-Fourier series of integrable functions was considered in 
\cite{bnpt}, where was used methods of martingale Hardy spaces. If we
consider the following restricted maximal operator $\widetilde{S}%
_{\#}^{\ast}f:=\sup_{n\in\mathbb{N}}\left\vert S_{M_n}f\right\vert,$  we
have weak (1,1) type inequality for $f\in L_1(G_m).$  That is,
\begin{eqnarray*}
	\lambda \mu\left\{ \widetilde{S}_{\#}^{\ast}f>\lambda \right\} \leq
	c\left\Vert f\right\Vert_{1}, \ \ \ f\in L_1(G_m), \ \ \lambda>0.
\end{eqnarray*}
Hence, if $f\in L_1(G_m),$ then $S_{M_n}f\to f, \ \text{a.e. on } \  G_m.$
Moreover, for any integrable function it is known that a.e. point is
Lebesgue point and for any such point $x$ of integrable function $f$ we
have that 
\begin{equation}  \label{smnvl}
S_{M_n}f(x)\to f(x), \ \text{as} \ n\to\infty, \ \text{for any such point} \ x
\ \text{of } \ f\in L_1(G_m).
\end{equation}

In the one-dimensional case Yano \cite{Yano} proved that 
\begin{equation*}
\left\Vert \sigma _{n}f-f\right\Vert _{p}\rightarrow 0,\text{ \ \ \ as \ \
	\ \ }n\rightarrow \infty ,\text{ \ }(f\in L_p(G_m),\text{ \ }1\leq p\leq
\infty ).
\end{equation*}

If we consider the maximal operator of F\'ejer means
$$\sigma^{\ast}f:=\sup_{n\in\mathbb{N}}\left\vert \sigma_{n}f\right\vert,  
$$
then 
\begin{eqnarray*}
	\lambda \mu\left\{ \sigma ^{\ast}f>\lambda \right\} \leq c\left\Vert
	f\right\Vert_{1}, \ \ \ f\in L_1(G_m), \ \ \lambda>0.
\end{eqnarray*}
This result can be found in Zygmund \cite{Zy} for the trigonometric series,
in Schipp \cite{Sc} (see also \cite{s1,s2} and \cite{BT1,GNT,pt,PTT,PTT2,tep1,tep10,tep11,tep12,tep13}) for Walsh series and in P\'al, Simon \cite{PS} for bounded Vilenkin series. The boundedness does not hold from Lebesgue space $L_1(G_m)$ to the space $
L_1(G_m)$. The  $\text{weak}-(1,1)$ type inequality  follows that for any $f\in
L_1(G_m),$ 
\begin{equation*}
\sigma_nf(x)\to f(x), \ \ \ \text{a.e., \ \ \ as} \ \ \ n\to\infty.
\end{equation*}
Moreover (for details see \cite{goggog}), for any integrable function it is
known that a.e. point is Vilenkin-Lebesgue points and for any such point $x$
of integrable function $f$ we have that $\sigma_nf(x)\to f(x), \ \ \text{as}
\ \ n\to\infty.$

It is also well-known (for details see \cite{sws}) that the maximal operator  $\sigma ^{\alpha,\ast }$ of Ces?ro means is bounded from the Lebesgue space $L_{1}$ to the space weak-$L_{1}$. It follows that $\sigma ^{\alpha }_nf(x)\to f(x), \ \ \ \text{a.e., as} \ \ \ n\to\infty$ for any $f\in L_1(G_m).$ The maximal operator 
$\sigma^{\alpha ,\ast }$ $\left( 0<\alpha <1\right) $  with respect to Vilenkin systems was also investigated by Weisz \cite{We1}
(see also \cite{BT,BTT} and \cite{We3,We2}).

In \cite{Ga2} Gát and Goginava proved some convergence and divergence
properties of the N\"orlund logarithmic means of functions in the Lebesgue
space $L_{1}.$ In particular,  (see also \cite{BPT,PTW2,tep2,TT1}) they 
proved that there exists an function in the space $L_1, $ such that $
\sup_{n\in \mathbb{N}}\left\Vert L_nf\right\Vert_1=\infty.$ However,
Goginava \cite{gog2} proved that $\left\Vert L_{2^n}f\right\Vert _{1}\leq
c\left\Vert f\right\Vert _{1}, \ \ f\in L_1, \ \ n\in \mathbb{N}. $
Moreover, if we consider restricted maximal operator 
\begin{equation*}
\widetilde{L}_{\#}^{\ast}f:=\sup_{n\in\mathbb{N}}\left\vert
L_{M_n}f\right\vert
\end{equation*}
of N\"orlund means, then Goginava \cite{gog2} proved that 
\begin{eqnarray*}
	\lambda \mu\left\{\widetilde{L}_{\#}^{\ast}f>\lambda \right\} \leq
	c\left\Vert f\right\Vert_{1}, \ \ \ f\in L_1(G_m), \ \ \lambda>0.
\end{eqnarray*}
It follows that for any $f\in L_1(G_m),$ 
$
L_{M_n}f(x)\to f(x), \ \ \ \text{a.e., as} \ \ \ n\to\infty.
$

M\'oricz and Siddiqi \cite{Mor} investigate  approximation properties of
some special N\"orlund means of Walsh-Fourier series of $L_{p}$ functions in
norm. Similar results for the two-dimensional case can be found in Nagy \cite{nagy,n}, Nagy and Tephnadze \cite{NT1,NT2,NT3,NT4}. Approximation properties of some general summability methods can be found in \cite{BPTW,BT2,BN,BNT,BPTW,MPT} (see also \cite{GT1,GT2}). Fridli, Manchanda and Siddiqi \cite{FMS} improved and
extended results of M\'oricz and Siddiqi \cite{Mor} to Martingale Hardy
spaces. The almost everywhere convergence of N\"orlund means of Vilenkin-Fourier series with
monotone coefficients of integrable functions was proved in \cite{PTW}.

In this paper we derive convergence of N\"orlund means of Vilenkin-Fourier
series with monotone coefficients of integrable functions in Lebesgue and
Vilinkin-Lebesgue points.

The paper is organized as following: In Section 3 we present and prove some auxiliary lemmas and in Section 4 we present and prove our main results. Moreover, in order not to disturb our discussions in these sections some preliminaries are
given in Section 2.

\section{Preliminaries}

Denote by $\mathbb{N}_{+}$ the set of the positive integers, $\mathbb{N}:=%
\mathbb{N}_{+}\cup \{0\}.$ Let $m:=(m_{0,}$ $m_{1},...)$ be a sequence of
the positive integers not less than 2. Denote by 
\begin{equation*}
Z_{m_{k}}:=\{0,1,...,m_{k}-1\}
\end{equation*}
the additive group of integers modulo $m_{k}$.

Define the group $G_{m}$ as the complete direct product of the groups $%
Z_{m_{i}}$ with the product of the discrete topologies of $Z_{m_{j}}`$s. In this paper we discuss bounded Vilenkin groups, i.e. the case
when $\sup_{n}m_{n}<\infty .$ The direct product $\mu $ of the measures 
\begin{equation*}
\mu _{k}\left( \{j\}\right) :=1/m_{k}\text{ \ \ \ }(j\in Z_{m_{k}})
\end{equation*}%
is the Haar measure on $G_{m_{\text{ }}}$with $\mu \left( G_{m}\right) =1.$

The elements of $G_{m}$ are represented by sequences 
\begin{equation*}
x:=\left( x_{0},x_{1},...,x_{j},...\right) ,\ \left( x_{j}\in
Z_{m_{j}}\right) .
\end{equation*}

It is easy to give a base for the neighborhood of $G_{m}:$
\begin{equation*}
I_{0}\left( x\right) :=G_{m},\text{ \ }I_{n}(x):=\{y\in G_{m}\mid
y_{0}=x_{0},...,y_{n-1}=x_{n-1}\},
\end{equation*}%
where $x\in G_{m},$ $n\in\mathbb{N}.$ Denote $I_{n}:=I_{n}\left( 0\right) $
for $n\in \mathbb{N}_{+},$ and $\overline{I_{n}}:=G_{m}$ $\backslash $ $I_{n}
$.

\bigskip If we define the so-called generalized number system based on $m$
in the following way $M_{0}:=1,\ M_{k+1}:=m_{k}M_{k}\,\,\,\ \ (k\in\mathbb{N}%
), $ then every $n\in\mathbb{N}$ can be uniquely expressed as $%
n=\sum_{j=0}^{\infty }n_{j}M_{j},$ where $n_{j}\in Z_{m_{j}}$ $(j\in\mathbb{N%
}_{+})$ and only a finite number of $n_{j}`$s differ from zero.

Next, we introduce on $G_{m}$ an orthonormal system which is called the
Vilenkin system. First, we define the complex-valued function $%
r_{k}\left( x\right) :G_{m}\rightarrow\mathbb{C},$ the generalized
Rademacher functions, by 
\begin{equation*}
r_{k}\left( x\right) :=\exp \left( 2\pi ix_{k}/m_{k}\right) ,\text{ }\left(
i^{2}=-1,x\in G_{m},\text{ \ }k\in\mathbb{N}\right) .
\end{equation*}

Now, define the Vilenkin system$\,\,\,\psi :=(\psi _{n}:n\in\mathbb{N})$ on $%
G_{m}$ as: 
\begin{equation*}
\psi _{n}(x):=\prod\limits_{k=0}^{\infty }r_{k}^{n_{k}}\left( x\right)
,\,\,\ \ \,\left( n\in\mathbb{N}\right) .
\end{equation*}

Specifically, we call this system the Walsh-Paley system when $m\equiv 2.$

The norms (or quasi-norms) of the spaces $L_{p}(G_{m})$ and $%
weak-L_{p}\left( G_{m}\right) $ $\left( 0<p<\infty \right) $ are
respectively defined by 
\begin{equation*}
\left\Vert f\right\Vert _{p}^{p}:=\int_{G_{m}}\left\vert f\right\vert
^{p}d\mu,\ \ \ \ \ \ \ \left\Vert f\right\Vert _{weak-L_{p}}^{p}:=\underset{%
	\lambda >0}{\sup }\lambda ^{p}\mu \left( f>\lambda \right) <+\infty .
\end{equation*}

The Vilenkin system is orthonormal and complete in $L_{2}\left( G_{m}\right) 
$ (see \cite{AVD}).

If $f\in L_{1}\left( G_{m}\right) $ we can define Fourier coefficients,
partial sums of the Fourier series, Dirichlet kernels in the usual manner: 
\begin{equation*}
\widehat{f}\left( n\right) :=\int_{G_{m}}f\overline{\psi }_{n}d\mu \ \ \ \ \
\, S_{n}f:=\sum_{k=0}^{n-1}\widehat{f}\left( k\right) \psi _{k},\text{ \ \ \ 
} D_{n}:=\sum_{k=0}^{n-1}\psi _{k\text{ }},\text{ \ \ }\left( n\in \mathbb{N}%
_{+}\right)
\end{equation*}
respectively. Recall that %\begin{equation*}
%D_{M_{n}}\left( x\right) =\left\{
%\begin{array}{ll}
%M_{n}, & \text{if\thinspace \thinspace \thinspace }x\in I_{n}, \\
%0, & \text{if}\,\,x\notin I_{n}.%
%\end{array}%
%\right.  \label{1dn}
%\end{equation*}
\begin{eqnarray}  \label{dn22}
&&\int_{G_{m}}D_n(x)dx=1, \\
&&D_{M_n-j}(x) =D_{M_n}(x)-\psi_{M_n-1}(x)\overline{D}_j(x), \ j<M_n. \label{dn23}
\end{eqnarray}

The convolution of two functions $f,g\in L_{1}(G_m)$ is defined by 
\begin{equation*}
\left( f\ast g\right) \left( x\right) :=\int_{G_m}f\left( x-t\right) g\left(
t\right) dt\text{ \ \ }\left( x\in G_m\right).
\end{equation*}%
It is easy to see that if $f\in L_{p}\left( G_m\right) ,$ $g\in  L_{1}\left(
G_m\right) $ and $1\leq p<\infty .$ Then $f\ast g\in  L_{p}\left( G_m\right) 
$ and 
\begin{equation}  \label{concond}
\left\Vert f\ast g\right\Vert_{p}\leq \left\Vert f\right\Vert_{p}\left\Vert
g\right\Vert _{1}.
\end{equation}

Let $\{q_{k}:k\geq 0\}$ be a sequence of non-negative numbers. Nörlund and $T$ means for a Fourier series of $f$ \ are respectively defined
by 
\begin{equation*}
t_{n}f=\frac{1}{Q_{n}}\overset{n}{\underset{k=1}{\sum }}q_{n-k}S_{k}f \ \ 
\text{and} \ \ T_nf:=\frac{1}{Q_n}\overset{n-1}{\underset{k=0}{\sum }}%
q_{k}S_kf, \ \ \text{where} \ \ Q_{n}:=\sum_{k=0}^{n-1}q_{k}.
\end{equation*}
It is obvious that 
\begin{equation*}
t_nf\left(x\right)=\underset{G_m}{\int}f\left(t\right)F_n\left(x-t\right)
d\mu\left(t\right) \ \ \ \ \text{and} \ \ \ \ T_nf\left(x\right)=\underset{%
	G_m}{\int}f\left(t\right)F^{-1}_n\left(x-t\right) d\mu\left(t\right)
\end{equation*}
where $F_n$ and $F^{-1}_n$ are kernels of N\"orlund and $T$ kernels
respectively
\begin{equation*}
F_n:=\frac{1}{Q_n}\overset{n}{\underset{k=1}{\sum }}q_{n-k}D_k \ \ \text{
	and } \ \ F^{-1}_n:=\frac{1}{Q_n}\overset{n}{\underset{k=1}{\sum }}q_{k}D_k.
\end{equation*}

We always assume that $\{q_k:k\geq 0\}$ be a sequence of non-negative
numbers and $q_0>0.$ Then N\"orlund means generated by $\{q_k:k\geq 0\}$ is
regular if and only if  
$\frac{q_{n-1}}{Q_n}\to 0, \  \text{as}  \ n\to \infty.$
Analogical regularity condition for  $T$  means is condition  $\lim_{n\rightarrow\infty}Q_n=\infty.$

If we invoke Abel transformation we get the following identities, which are
very important for the investigations of N\"orlund summability: 
\begin{eqnarray}  \label{2b}
Q_n&:=&\overset{n-1}{\underset{j=0}{\sum}}q_j=\overset{n}{\underset{j=1}{%
		\sum }}q_{n-j}\cdot 1 =\overset{n-1}{\underset{j=1}{\sum}}%
\left(q_{n-j}-q_{n-j-1}\right) j+q_0n,
\end{eqnarray}
\begin{equation}  \label{2bb}
F_n=\frac{1}{Q_n}\left(\overset{n-1}{\underset{j=1}{\sum}}\left(
q_{n-j}-q_{n-j-1}\right) jK_{j}+q_0nK_n\right)
\end{equation}
and
\begin{equation}  \label{2bbb}
t_n=\frac{1}{Q_n}\left(\overset{n-1}{\underset{j=1}{\sum}}\left(
q_{n-j}-q_{n-j-1}\right) j\sigma_{j}+q_0n\sigma_n\right).
\end{equation}

Let consider some class of N\"orlund means with monotone and bounded sequence $%
\{q_k:k\in \mathbb{N}\}$, such that $q:=\lim_{n\rightarrow\infty}q_n>c>0.$
It easy to check that 
$\frac{q_{n-1}}{Q_n}=O\left(\frac{1}{n}\right),\text{ \ as \ } n\rightarrow
\infty.$

The well-known example of N\"orlund summability is $\left(C,\alpha\right)$
-means (Ces\`aro means), where $0<\alpha<1$  which are defined by 
\begin{equation*}
\sigma_n^{\alpha}f:=\frac{1}{A_n^{\alpha}}\overset{n}{\underset{k=1}{ \sum}}%
A_{n-k}^{\alpha-1}S_kf, \ \text{where } \ A_0^{\alpha}:=0, \ A_n^{\alpha}:=%
\frac{\left(\alpha+1\right)...\left(\alpha+n\right)}{n!}.
\end{equation*}

Let $V_n^{\alpha}$ denote the N\"orlund mean, where 
$
\left\{q_0=1,\ q_k=k^{\alpha-1}:k\in \mathbb{N}_+\right\},
$
that is 
\begin{equation*}
V_n^{\alpha}f:=\frac{1}{Q_n}\overset{n-1}{\underset{k=0}{\sum }}%
(n-k)^{\alpha-1}S_kf,\qquad 0<\alpha<1.
\end{equation*}
It is easy to show that
$
\frac{q_{n-1}}{Q_n}= O\left(\frac{1}{n}\right)\rightarrow
0,\text{ \ as \ }n\rightarrow \infty.
$

The $n$-th Riesz`s logarithmic mean $R_{n}$ and  Nörlund logarithmic mean 
$L_{n}$ are defined by 
\begin{equation*}
R_{n}f:=\frac{1}{l_{n}}\sum_{k=1}^{n-1}\frac{S_{k}f}{k},\text{ \ \ \ }
L_{n}f:=\frac{1}{l_{n}}\sum_{k=1}^{n-1}\frac{S_{k}f}{n-k}, \ \ \text{where}
\ \ l_{n}:=\sum_{k=1}^{n-1}1/k.
\end{equation*}

Up to now we have considered $T$ mean in the case when the sequence $%
\{q_k:k\in\mathbb{N}\}$ is bounded but now we consider $T$ summabilities
with unbounded sequence $\{q_k:k\in\mathbb{N}\}$. 

Let us define the class of N\"orlund means with non-decreasing coefficients: 
\begin{equation*}
\beta_n^{\alpha}f:=\frac{1}{Q_n}\sum_{k=1}^{n}\log^{\alpha}\left( n-k-1\right)S_kf.
\end{equation*}%

It is obvious that
$
\frac{n}{2}\log^{\alpha}\left(n/2\right)\leq Q_n\leq n\log^{\alpha}n.
$
It follows that
\begin{eqnarray*} \label{node00}
	\frac{1}{Q_n}= O\left(\frac{1}{n}\right)\rightarrow 0,\text{ \ as \ }n\rightarrow \infty \ \ \ \ \ 
	\text{and}
	\ \ \ \ \ 
	\frac{q_{n-1}}{Q_n}= O\left(\frac{1}{n}\right)\rightarrow 0,\text{ \ as \ }n\rightarrow \infty.
\end{eqnarray*}

A point $x\in G_m$ is called a Lebesgue point of integrable function $f$ if
\begin{eqnarray*}
	\lim_{n\to 0}\frac{1}{\vert I_n(x)\vert}\int_{I_n(x)}\left\vert f(t)-f\left( x\right)\right\vert d\mu(t)=0.
\end{eqnarray*}

It is known that a.e. point $x\in G_m$ is a Lebesgue point of function $f$ and Fej\'{e}r means $\sigma_nf$ of trigonometric Fourier series of $f\in L_1(G_m)$ converge to $f$ at each Lebesgue point. It is also known that if $x\in G_m$ is point of continuity of function $f$ then it is Lebesgue point.

Let introduced the operator
\begin{eqnarray*}
	W_nf(x):= \sum_{s=0}^{n-1}M_s\sum_{r_s=1}^{m_s-1}\int_{I_n(x-r_se_s)}\left\vert f(t)-f\left( x\right)\right\vert d\mu(t)
\end{eqnarray*}

A point $x \in G_m $ is a Vilenkin-Lebesgue point of $f \in L_1 (G_m),$ if
$$
\lim_{n \rightarrow \infty} W_nf (x)=0.
$$

\section{Auxiliary Lemmas}

Next two lemmas can be found in \cite{AVD}:

\begin{lemma}\label{lemma7kn0}
	Let $n\in \mathbb{N}.$ Then for some constant $c,$ we have
	\begin{equation} \label{fn5}
	n\left\vert K_n\right\vert\leq c\sum_{l=\left\langle n\right\rangle}^ {\left\vert n\right\vert}M_l\left\vert K_{M_l}\right\vert\leq
	c\sum_{l=0}^{\left\vert n\right\vert }M_l\left\vert K_{M_l}\right\vert.
	\end{equation}
\end{lemma}

\begin{lemma}\label{lemma7kn}
	Let $n\in \mathbb{N}.$ Then, for any $n,N\in \mathbb{N_+}$,
	\begin{eqnarray} \label{fn40}
	&&\int_{G_m} K_n (x)d\mu(x)=1,\\
	&& \label{fn4}
	\sup_{n\in\mathbb{N}}\int_{G_m}\left\vert K_n(x)\right\vert d\mu(x)<\infty,\\
	&& \label{fn400}
	\sup_{n\in\mathbb{N}}\int_{G_m \backslash I_N}\left\vert K_n(x)\right\vert d\mu (x)\rightarrow  0, \ \ \text{as} \ \ n\rightarrow  \infty.
	\end{eqnarray}
\end{lemma}

First we consider kernels of N\"orlund kernels with non-decreasing sequences:
\begin{lemma}\label{lemma0nn}
	Let $\{q_k:k\in\mathbb{N}\}$ be a sequence of non-decreasing numbers, satisfying the condition
	\begin{equation} \label{fn01}
	\frac{q_{n-1}}{Q_n}=O\left( \frac{1}{n}\right) ,\text{ \ \ as \ \ }
	n\rightarrow\infty.
	\end{equation}
	Then, for some constant $c,$ we have
	\begin{equation*}
	\left\vert F_n\right\vert\leq\frac{c}{n}\left\{\sum_{j=0}^{\left\vert n\right\vert }M_j\left\vert K_{M_j}\right\vert \right\}.
	\end{equation*}
\end{lemma}
\begin{proof}
	Let the sequence $\{q_k:k\in \mathbb{N}\}$ be non-decreasing. Then, by using (\ref{fn01}), we get that
	\begin{eqnarray*}
		\frac{1}{Q_n}\left(\overset{n-1}{\underset{j=1}{\sum }}\left\vert
		q_{n-j}-q_{n-j-1}\right\vert+q_0\right) 
		\leq\frac{1}{Q_n}\left(\overset{n-1}{\underset{j=1}{\sum }}\left(
		q_{n-j}-q_{n-j-1}\right)+q_0\right) 
		\leq \frac{q_{n-1}}{Q_{n}}\leq \frac{c}{n}.
	\end{eqnarray*}
	
Hence, in view of \eqref{fn01} if we apply Lemma \ref{lemma7kn0} and use the equalities (\ref{2b}) and (\ref{2bb}) we obtain that
\begin{eqnarray*}
	\left\vert F_n\right\vert \leq  \left( \frac{1}{Q_n}\left( \overset{n-1}{\underset{j=1}{\sum }}\left\vert q_{n-j}-q_{n-j-1} \right\vert+q_0\right)\right)\sum_{i=0}^{\left\vert n\right\vert } M_i\left\vert K_{M_i}\right\vert \leq \frac{c}{n}\sum_{i=0}^{\left\vert n\right\vert }M_i\left\vert K_{M_i}\right\vert.
\end{eqnarray*}
The proof is complete.
\end{proof}

\begin{corollary}\label{Corollary3nn} Let $\{q_k:k\in \mathbb{N}\}$ be a sequence of non-decreasing numbers. Then, for any $n, N\in \mathbb{N_+},$  
\begin{eqnarray} \label{1.71}
&&\int_{G_m} F_n(x) d\mu (x)=1, \\
&&\sup_{n\in\mathbb{N}}\int_{G_m}\left\vert F_n(x)\right\vert d\mu(x)<\infty,\label{1.72} \\
&&\sup_{n\in\mathbb{N}}\int_{G_m \backslash I_N}\left\vert F_n(x)\right\vert d\mu (x)\rightarrow  0, \ \ \text{as} \ \ n\rightarrow  \infty. \label{1.73}
\end{eqnarray}
\end{corollary}
\begin{proof}
According to \eqref{dn22} we readily obtain \eqref{1.71}. By using \eqref{fn4} in Lemma \ref{lemma7kn}, combined with \eqref{2b} and \eqref{2bb} we get that
\begin{eqnarray*}
\int_{G_m}\left\vert F_n(x)\right\vert d\mu(x)&\leq&\frac{1}{Q_n}\overset{n-1}{\underset{j=1}{\sum}}\left(q_{n-j}-q_{n-j-1}\right)j\int_{G_m}\left\vert K_j\right\vert d\mu +\frac{q_0n}{Q_n} \int_{G_m}\vert K_n\vert d\mu\\
&\leq&\frac{c}{Q_n}\overset{n-1}{\underset{j=1}{\sum}}\left(q_{n-j}-q_{n-j-1}\right)j+\frac{cq_0n}{Q_n}<c<\infty.
\end{eqnarray*}
By using \eqref{fn400} in Lemma \ref{lemma7kn} and inequalities \eqref{2b} and \eqref{2bb} we can conclude that
\begin{eqnarray*}
&&\int_{G_m \backslash I_N}\left\vert F_n\right\vert d\mu 
\leq\frac{1}{Q_n}\overset{n-1}{\underset{j=0}{\sum}}\left(q_{n-j}-q_{n-j-1}\right)j\int_{G_m \backslash I_N}\left\vert K_j\right\vert d\mu  +\frac{q_0n}{Q_n}\int_{G_m \backslash I_N}\vert K_n\vert d\mu\\
&&\leq\frac{1}{Q_n}\overset{n-1}{\underset{j=0}{\sum}}\left(q_{n-j}-q_{n-j-1}\right)j\alpha_j+\frac{q_0n\alpha_n}{Q_n}:=I+II, \ \text{ where }  \  \alpha_n\to 0, \  \text{as} \  n\to\infty. \notag
\end{eqnarray*}
Since the sequence is non-decreasing, we can conclude that	$II\leq \alpha_n\to 0, \ \ \text{as} \ \ n\to\infty.$ 
	
On the other hand, since $\alpha_n$  converges to $0,$  we get that there exists an absolute constant $A,$ such that $\alpha_n\leq A$ for any $n\in \mathbb{N}$ and for any $\varepsilon>0$ there exists $N_0\in \mathbb{N},$ such that $\alpha_n< \varepsilon$ when $n>N_0.$ Hence, 
	\begin{eqnarray*}
		I&=&\frac{1}{Q_n}\overset{N_0}{\underset{j=1}{\sum}}\left(q_{n-j}-q_{n-j-1}\right)j\alpha_j
		+\frac{1}{Q_n}\overset{n-1}{\underset{j=N_0+1}{\sum}}\left(q_{n-j}-q_{n-j-1}\right)j\alpha_j :=I_1+I_2
	\end{eqnarray*}
	Since  	$\vert q_{n-j}-q_{n-j-1}\vert<2q_{n-1}$ and $\alpha_n<A,$ we obtain that
	\begin{eqnarray*}
	I_1=\frac{1}{Q_n}\overset{N_0}{\underset{j=1}{\sum}}\left(q_{n-j}-q_{n-j-1}\right)j\alpha_j \leq \frac{2AN_0q_{n-1}}{Q_n}\to 0, \ \ \ \text{as} \ \ \ n\to \infty
	\end{eqnarray*}
	and
	\begin{eqnarray*}
		I_2&=&\frac{1}{Q_n}\overset{n-1}{\underset{j=N_0+1}{\sum}}\left(q_{n-j}-q_{n-j-1}\right)j\alpha_j \\
		&\leq& \frac{\varepsilon}{Q_n}\overset{n-1}{\underset{j=N_0+1}{\sum}}\left(q_{n-j}-q_{n-j-1}\right)j
		\leq \frac{\varepsilon}{Q_n}\overset{n-1}{\underset{j=0}{\sum}}\left(q_{n-j}-q_{n-j-1}\right)j<\varepsilon.
	\end{eqnarray*}
We conclude that also $I_2\to 0,$ so  the proof is complete.
\end{proof}

\begin{lemma} \label{lemma0nn1}
Let $\{q_k:k\in\mathbb{N}\}$ be a sequence of non-increasing numbers satisfying the condition
\begin{equation} \label{fn0}
\frac{1}{Q_n}=O\left(\frac{1}{n}\right),\text{\ \ as \ \ }n\rightarrow \infty.
\end{equation}
Then, for some constant $c,$ we have
	\begin{equation*}
	\left\vert F_n\right\vert \leq\frac{c}{n}\left\{\sum_{j=0}^{\left\vert n\right\vert }M_j\left\vert K_{M_j}\right\vert\right\}.
	\end{equation*}
\end{lemma}
\begin{proof}
	Let the sequence $\{q_k:k\in \mathbb{N}\}$ be non-increasing and satisfying
	condition (\ref{fn0}). Then
	\begin{eqnarray*}
		&&\frac{1}{Q_n}\left( \overset{n-1}{\underset{j=1}{\sum}}\left\vert
		q_{n-j}-q_{n-j-1}\right\vert+q_0\right) \\
		&\leq& \frac{1}{Q_n}\left(\overset{n-1}{\underset{j=1}{\sum}}-
		\left(q_{n-j}-q_{n-j-1}\right)+q_0\right) \\
		&\leq&\frac{2q_0-q_{n-1}}{Q_n}\leq\frac{2q_0}{Q_n}\leq\frac{c}{n}.
	\end{eqnarray*}
	
	If we apply Lemma \ref{lemma7kn0} and invoke equalities (\ref{2b}) and (\ref{2bb}) we immediately get that
	\begin{eqnarray*}
		\left\vert F_n\right\vert \leq\left(\frac{1}{Q_n} \left(\overset{n-1}{\underset{j=1}{\sum }}\left\vert q_{n-j}-q_{n-j-1}\right\vert +q_0\right)\right) \sum_{i=0}^{\left\vert n\right\vert }M_{i}\left\vert K_{M_i}\right\vert
		\leq \frac{c}{n}\sum_{i=0}^{\left\vert n\right\vert }M_i\left\vert
		K_{M_i}\right\vert.
	\end{eqnarray*}
	
	The proof is complete.	
\end{proof}

\begin{corollary}
	\label{corollary3n9} Let $\{q_k:k\in \mathbb{N}\}$ be a sequence of non-increasing numbers satisfying the condition \eqref{fn0}. Then, for any $n,N\in \mathbb{N_+}$,
	\begin{eqnarray} \label{1.71inc}
	&&\int_{G_m} F_n(x) d\mu (x)=1, \\
	&&\sup_{n\in\mathbb{N}}\int_{G_m}\left\vert F_n(x)\right\vert d\mu(x)<\infty,\label{1.72inc} \\
	&&\sup_{n\in\mathbb{N}}\int_{G_m \backslash I_N}\left\vert F_n(x)\right\vert d\mu (x)\rightarrow  0, \ \ \text{as} \ \ n\rightarrow  \infty, \ \ \text{for any} \ \ N\in \mathbb{N_+}.\label{1.73inc}
	\end{eqnarray}
\end{corollary}
\begin{proof}
	If we compare the estimation of $F_n$ in Corollary \ref{lemma0nn1} with the estimation of $F_n$ in Lemma \ref{lemma0nn}  we find that they are quite the same. Hence, the proof is analogous to the proof of Corollary \ref{Corollary3nn}, so, we leave out the details.
\end{proof}

Finally, we study special subsequences of kernels of N\"orlund and $T$ means:
\begin{lemma}\label{lemma0nnT121}
Let $n\in \mathbb{N}$. Then
\begin{eqnarray*}  F_{M_n}(x)=D_{M_n}(x)-\psi_{M_n-1}(x)\overline{F^{-1}}_{M_n}(x)
\end{eqnarray*}
\end{lemma}
\begin{proof} By using \eqref{dn23} we get that
\begin{eqnarray*}
F_{M_n}(x)&=&\frac{1}{Q_{M_n}}\overset{M_n}{\underset{k=1}{\sum}}q_{M_n-k}D_k(x)=\frac{1}{Q_{M_n}}\overset{M_n-1}{\underset{k=0}{\sum }}q_kD_{M_n-k}(x)\\
&=&\frac{1}{Q_{M_n}}\overset{M_n-1}{\underset{k=0}{\sum }}q_k\left(D_{M_n}(x)-\psi_{M_n-1}(x)\overline{D}_j(x)\right)=D_{M_n}(x)-\psi_{M_n-1}(x)\overline{F^{-1}}_{M_n}(x).
\end{eqnarray*}
The proof is complete.
\end{proof}

\begin{lemma}\label{lemma0nnT11}
	Let $\{q_k:k\in\mathbb{N}\}$ be a sequence of non-increasing numbers. Then, for any $n,N\in \mathbb{N_+}$,
\begin{eqnarray} \label{1.71alphaT1}
&&\int_{G_m} F^{-1}_n(x) d\mu (x)=1, \\
&&\sup_{n\in\mathbb{N}}\int_{G_m}\left\vert F^{-1}_n(x)\right\vert d\mu(x)<\infty,\label{1.72alphaT1} \\
&&\sup_{n\in\mathbb{N}}\int_{G_m \backslash I_N}\left\vert F^{-1}_n(x)\right\vert d\mu (x)\rightarrow  0, \ \ \text{as} \ \ n\rightarrow  \infty, \label{1.73alphaT1}
\end{eqnarray}
\end{lemma}
\begin{proof}
If we follow analogous steps of Corollaries \ref{Corollary3nn} and \ref{corollary3n9} we immediately get proof. So, we leave out the details.
\end{proof}

\begin{corollary} \label{corollary3n}
	Let $\{q_k:k\in \mathbb{N}\}$ be a sequence of non-increasing numbers. Then, for any $n,N\in \mathbb{N_+}$,
	\begin{eqnarray} \label{1.71alpha}
	&&\int_{G_m} F_{M_n}(x) d\mu (x)=1, \\
	&&\sup_{n\in\mathbb{N}}\int_{G_m}\left\vert F_{M_n}(x)\right\vert d\mu(x)\leq c<\infty,\label{1.72alpha} \\
	&&\sup_{n\in\mathbb{N}}\int_{G_m \backslash I_N}\left\vert F_{M_n}(x)\right\vert d\mu (x)\rightarrow  0, \ \ \text{as} \ \ n\rightarrow  \infty, \label{1.73alpha}
	\end{eqnarray}
\end{corollary}
\begin{proof}
The proof immediately follows Lemmas \ref{lemma0nnT121} and \ref{lemma0nnT11}.
\end{proof}

\section{Proofs of the Theorems}

\begin{theorem}\label{Corollary3nnconv} a) Let $p\geq 1$ and $\{q_k:k\in \mathbb{N}\}$ be a sequence of non-decreasing numbers. Then, 	for all  $f\in L_p(G_m)$,
	$$\Vert t_n f-f\Vert_p \to 0 \ \ \text{as}\ \ n\to \infty.$$
	
	Let function $f\in L_1(G_m)$ is continuous at a point $x.$ Then
	${{t}_{n}}f(x)\to f(x), \ \ \ \text{as}\ \  \ n\to\infty.$
	Moreover, for all Vilenkin-Lebesgue points of $f\in L_p(G_m)$,
	\begin{equation*}
	\underset{n\rightarrow \infty }{\lim }t_nf(x)=f(x).
	\end{equation*}

b)	Let $p\geq 1$ and $\{q_k:k\in \mathbb{N}\}$ be a sequence of non-increasing numbers satisfying the condition \eqref{fn0}. Then, for all  $f\in L_p(G_m)$, 
	$$\Vert t_n f-f\Vert_p \to 0 \ \ \text{as}\ \ n\to \infty.$$
		
Let function $f\in L_1(G_m)$ is continuous at a point $x.$ Then
${{t}_{n}}f(x)\to f(x), \ \ \ \text{as}\ \  \ n\to\infty.$ Moreover, for all Vilenkin-Lebesgue points of $f\in L_p(G_m)$,
	\begin{equation*}
	\underset{n\rightarrow \infty }{\lim }t_nf(x)=f(x).
	\end{equation*}
\end{theorem}
\begin{proof} Let $\{q_k:k\in \mathbb{N}\}$ be a non-decreasing sequence.
Corollary \ref{Corollary3nn} immediately follows the stated norm  and pointwise convergences. Suppose that $x$ is either a point of continuity or Vilenkin-Lebesgue point of the function $f\in L_p(G_m).$ Then
$
\underset{n\rightarrow \infty }{\lim }\vert\sigma_nf(x)-f(x)\vert=0.
$
Hence, by combining \eqref{2b} and \eqref{2bbb} we can conclude that

	\begin{eqnarray*}
		&&\vert t_nf(x)-f(x)\vert \\
		&\leq&\frac{1}{Q_n}\left(\overset{n-2}{\underset{j=1}{\sum}}\left(q_{n-j}-q_{n-j-1}\right)j\vert\sigma_jf(x)-f(x)\vert+q_0n\vert\sigma_nf(x)-f(x)\vert\right)\\
		&\leq&\frac{1}{Q_n}\overset{n-2}{\underset{j=0}{\sum}}\left(q_{n-j}-q_{n-j-1}\right)j\alpha_j+\frac{q_0n\alpha_n}{Q_n}:=I+II, \  \text{	where  } \  \alpha_n\to 0, \  \text{as} \  n\to\infty.
	\end{eqnarray*}
To prove $I\to 0, \ \text{as} \ n\to\infty$ and $II\to 0, \ \text{as} \ n\to\infty,$ we just have to use analogous steps of Corollary \ref{Corollary3nn}. It follows that Part a) is proved.

Let sequence is non-increasing, satisfying condition \eqref{fn0}. 
According to Corollary \ref{corollary3n9} we get norm convergence and pointwise  convergence.
To prove  convergence in Vilenkin-Lebesgue points we use estimations \eqref{2b} and \eqref{2bbb} to obtain that
\begin{eqnarray*}
\vert t_nf-f(x)\vert \leq\frac{1}{Q_n}\overset{n-2}{\underset{j=0}{\sum}}\left(q_{n-j-1}-q_{n-j}\right)j\alpha_j+\frac{q_0n\alpha_n}{Q_n} 
		=III+IV \  \text{	where  } \  \alpha_n\to 0, \  \text{as} \  n\to\infty.
\end{eqnarray*}

It is evident that
$IV\leq\frac{q_{0}n\alpha_n}{Q_n}\leq C \alpha_n\to 0, \ \ \text{as} \ \ n\to\infty.$	
Moreover, for any $\varepsilon>0$ there exists $N_0\in \mathbb{N},$ such that $\alpha_n< \varepsilon$ when $n>N_0.$ It follows that
	\begin{eqnarray*}
		&&\frac{1}{Q_n}\overset{n-2}{\underset{j=1}{\sum}}\left(q_{n-j-1}-q_{n-j}\right)j\alpha_j \\
		&=&\frac{1}{Q_n}\overset{N_0}{\underset{j=1}{\sum}}\left(q_{n-j-1}-q_{n-j}\right)j\alpha_j
		+\frac{1}{Q_n}\overset{n-2}{\underset{j=N_0+1}{\sum}}\left(q_{n-j-1}-q_{n-j}\right)j\alpha_j=III_1+III_2.
	\end{eqnarray*}
	Since sequence is non-increasing, we can conclude that 
	$\vert q_{n-j}-q_{n-j-1}\vert<2q_{0}.$  Hence,
	\begin{eqnarray*}
		III_1\leq\frac{2q_{0}N_0}{Q_n}\to 0, \ \ \ \text{as} \ \ \ n\to \infty
	\end{eqnarray*}
	and
\begin{eqnarray*}
III_2&\leq&\frac{1}{Q_n}\overset{n-2}{\underset{j=N_0+1}{\sum}}\left(q_{n-j-1}-q_{n-j}\right)j\alpha_j\leq\frac{\varepsilon(n-1)}{Q_n}\overset{n-2}{\underset{j=N_0+1}{\sum}}\left(q_{n-j}-q_{n-j-1}\right) \\
&\leq& \frac{\varepsilon(n-1)}{Q_n}\left(q_{0}-q_{n-N_0}\right)\leq \frac{2q_0\varepsilon(n-1)}{Q_n}<C\varepsilon.
\end{eqnarray*}
	The proof of part b)  is also complete.
\end{proof}
\begin{corollary}
	Let $f\in L_p,$ where  $p\geq 1.$ Then%
	\begin{eqnarray*}
\sigma_{n}f &\rightarrow& f,\text{ \ \ \ a.e., \ \ \  as \ }n\rightarrow \infty,\ \  \ \ \ \ \ \
		\sigma_n^{\alpha}f \rightarrow f,\text{ \ \ \  a.e., \ \ \ as \ } n\rightarrow
		\infty,\\ 
		V_{n}f &\rightarrow& f,\text{ \ \ \ a.e., \ \ \  as \ } n\rightarrow
		\infty ,\ \  \ \ \ \ \ \
		B^{\alpha,\beta}_{n}f \rightarrow f,\text{ \ \ \ a.e., \ \ \ as \ }n\rightarrow
		\infty. 
	\end{eqnarray*}
\end{corollary}

\begin{theorem}\label{Corollaryconv4} Let $p\geq 1$ and $\{q_k:k\in \mathbb{N}\}$ be a sequence of non-increasing numbers. Then, for all  $f\in L_p(G_m)$,
	$$\Vert t_{M_n} f-f\Vert_p \to 0 \ \ \text{as}\ \ n\to \infty$$
	Moreover, for all Lebesgue points of $f\in L_p(G_m)$,
	\begin{equation*}
	\underset{n\rightarrow \infty }{\lim }t_{M_n}f(x)=f(x).
	\end{equation*}
\end{theorem}
\begin{proof}
From Corollary \ref{lemma0nnT11} we immediately get norm convergence convergence. To prove a.e convergence we use  Lemma \ref{lemma0nnT121}  to write that
\begin{eqnarray*}
t_{M_n}f\left(x\right)&=&\underset{G_m}{\int}f\left(t\right)F_n\left(x-t\right) d\mu\left(t\right)\\
&=&\underset{G_m}{\int}f\left(t\right)D_{M_n}\left(x-t\right)d\mu\left(t\right)
-\underset{G_m}{\int}f\left(t\right)\psi_{M_n-1}(x-t)\overline{F^{-1}}_{M_n}(x-t)
=I-II.
\end{eqnarray*}
	By applying \eqref{smnvl} we get that $I=S_{M_n}f(x)\to f(x)$
	for all Lebesgue points of $f\in L_p(G_m)$.
	
According to $\psi_{M_n-1}(x-t)=\psi_{M_n-1}(x)\overline{\psi}_{M_n-1}(t)$ we can conclude that
	$$II=\psi_{M_n-1}(x)\underset{G_m}{\int}f\left(t\right)\overline{F^{-1}}_{M_n}(x-t)\overline{\psi}_{M_n-1}(t)d(t)$$
	By combining \eqref{concond} and Lemma \ref{lemma0nnT11} we find that function 
	$$f\left(t\right)\overline{F^{-1}}_{M_n}(x-t)\in L_p \ \ \text{ where} \ \  p\geq  1 \ \ \text{for any } \ \ x\in G_m, $$
	and $II$ is Fourier coefficients of integrable function. According to Riemann-Lebesgue Lemma we get that
	$II\to 0 \ \ \text{for any } \ \ x\in G_m.$
The proof is complete.
\end{proof}
\begin{corollary}
	Let $f\in L_1(G_m)$ is continuous at a point $x.$ Then
${{t}_{n}}f(x)\to f(x), \ \ \ \text{as}\ \  \ n\to\infty.$
\end{corollary}
\begin{corollary}
	Let $f\in L_{p}$, where  $p\geq 1.$  Then, for all Lebesgue points of $f,$
	\begin{eqnarray*}
		L_{M_n}f(x)\rightarrow f(x),\text{ \ \ as \ \ }n\rightarrow \infty,
\end{eqnarray*}
\end{corollary}
\begin{corollary}
	Let $f\in L_1(G_m)$ is continuous at a point $x.$ Then
	$L_nf(x)\to f(x), \ \ \text{as}\ \ n\to\infty.$
\end{corollary}


\begin{thebibliography}{99}
\bibitem{AVD} \textit{G. N. Agaev, N. Ya. Vilenkin, G. M. Dzhafarly and  A. I. Rubinshtein,} Multiplicative systems of functions and harmonic analysis on zero-dimensional groups, Baku, Ehim, 1981 (in Russian).


\bibitem{BPTW} \textit{L. Baramidze, L. E. Persson, G. Tephnadze and P. Wall,} Sharp $H_p- L_p$ type inequalities of weighted maximal operators of Vilenkin-Nörlund means and its applications, J. Inequal. Appl., 2016, DOI: 10.1186/s13660-016-1182-1. 

\bibitem{Billard1967} \textit{P. Billard}, Sur la convergence presque partout des s\'eries de Fourier-Walsh des fonctions de l'espace $L^2 [0,1]$, Studia Math., 28, (1967), 363-388.

\bibitem{BT1} \textit{I. Blahota and G. Tephnadze,} Strong convergence theorem for Vilenkin-Fejér means, Publ. Math. Debrecen, 85 (1-2) (2014), 181-196.

\bibitem{BT} \textit{I. Blahota and G. Tephnadze,} On the $(C,\alpha)$-means with respect to the Walsh system, Anal. Math., 40 (2014), 161-174.

\bibitem{BT2} \textit{I. Blahota and G. Tephnadze,} A note on maximal operators of Vilenkin-Nörlund means, Acta Math. Acad. Paed. Nyíreg., 32 (2016), 203-213. 

\bibitem{BTT} \textit{I. Blahota, G. Tephnadze and R. Toledo,} Strong convergence theorem of Cesaro means with respect to the Walsh system, Tohoku Math. J., 67, 4 (2015), 573-584.

\bibitem{BN} \textit{I. Blahota and K. Nagy, }  Approximation by $\Theta$-means of Walsh-Fourier series,  Anal. Math., 44(1), 57-71.

\bibitem{BNT} \textit{I. Blahota, K. Nagy and G. Tephnadze, } Approximation by Marcinkiewicz $\Theta$-means of double Walsh-Fourier series, Math. Inequal. Appl., 22 (2019), no. 3, 837--853.


\bibitem{bnpt} \textit{I. Blahota, K. Nagy, L. E. Persson, G. Tephnadze,} A sharp boundedness result concerning some maximal operators of partial sums with respect to Vilenkin systems, Georgian Math., J.,26, 3(2019), 351-360. 

\bibitem{BPT} \textit{I. Blahota, L.-E. Persson and G Tephnadze,} On the Nörlund means of Vilenkin-Fourier series, Czech. Math. J., 65 (4), 983-1002.

\bibitem{BPT1} \textit{I. Blahota, L. E. Persson and G. Tephnadze,} Two-sided estimates of the Lebesgue constants with respect to Vilenkin systems and applications, Glasgow Math. J., 60 (2018) 17-34.


\bibitem{FMS} \textit{S. Fridli, P. Manchanda and A.H. Siddiqi,} Approximation by Walsh-N\"orlund means, Acta Sci. Math.(Szeged) 74 (2008), no. 3-4, 593-608.

\bibitem{Ga2} \textit{G. G\`{a}t and U. Goginava,} Uniform and $L$-convergence
of logarithmic means of Walsh-Fourier series, Acta Math. Sin. 22 (2006), no.
2, 497--506.

\bibitem{gog2} \textit{U. Goginava,} Almost everywhere convergence of subsequence of logarithmic means of Walsh-Fourier series, Acta Math. Acad. Paed. Nyíreg., 21 (2005), 169-175.

\bibitem{goggog} \textit{U. Goginava and L. Gogoladze,} Pointwise summability of Vilenkin-Fourier series, Publ. Math. Debrecen, Vol 79, 1-2 (2011), 89-108.

\bibitem{GNT} \textit{N. Gogolashvili, K. Nagy and G. Tephnadze,} Strong convergence theorem for Walsh-Kaczmarz-Fej\'er means, Mediterr. J. Math., 2021, 18 (2), 37.

\bibitem{GT1} \textit{N. Gogolashvili and G. Tephnadze,} On the maximal operators of $T$ means with respect to Walsh-Kaczmarz system, Stud. Sci. Math. Hung., 2021, 58 (1), 119-135.

\bibitem{GT2} \textit{N. Gogolashvili and G. Tephnadze,} Maximal operators of $T$ means with respect to Walsh-Kaczmarz system, Math. Inequal. Appl., 24, 3 (2021) 737-750.

\bibitem{goles}  \textit{J. A. Gosselin,} Almost everywhere convergence of Vilenkin-Fourier series, Trans. Amer. Math. Soc., 185 (1973), 345-370.

\bibitem{LPTT} \textit{D. Lukkassen, L.E. Persson, G. Tephnadze and G. Tutberidze,} Some inequalities related to strong convergence of Riesz logarithmic means of Vilenkin-Fourier series, J. Inequal. Appl., 2020, DOI: https://doi.org/10.1186/s13660-020-02342-8.

\bibitem{MPT} \textit{N. Memić, L. E. Persson and G. Tephnadze,} A note on the maximal operators of Vilenkin-Nörlund means with non-increasing coefficients, Stud. Sci. Math. Hung., 53, 4, (2016) 545-556.

\bibitem{Mor} \textit{F. M\'oricz and A. Siddiqi}, Approximation by N\"orlund means of Walsh-Fourier series, J. Approx. Theory 70 (1992), no. 3, 375--389.

\bibitem{nagy} \textit{K. Nagy,} On the maximal operator of Walsh-Marcinkiewicz
means, Publ. Math. Debrecen, 78 (3-4) (2011), 633-646.

\bibitem{n} \textit{K. Nagy,} Approximation by N\"orlund means of quadratical partial sums of double Walsh-Fourier series. Anal. Math. 26 (2010),
299-319.

\bibitem{NT1} \textit{K. Nagy and G. Tephnadze,} Walsh-Marcinkiewicz means and Hardy spaces, Cent. Eur. J. Math., 12, 8 (2014), 1214-1228.

\bibitem{NT2}  \textit{K. Nagy and G. Tephnadze,} Approximation by Walsh-Marcinkiewicz means on the Hardy space , Kyoto J. Math., 54, 3 (2014), 641-652.

\bibitem{NT3} 	\textit{K. Nagy and G. Tephnadze,} Strong convergence theorem for Walsh-Marcinkiewicz means, Math. Inequal. Appl., 19, 1 (2016), 185-195.

\bibitem{NT4} \textit{K. Nagy and G. Tephnadze,} The Walsh-Kaczmarz-Marcinkiewicz means and Hardy spaces, Acta math. Hung., 149, 2 (2016), 346-374.

\bibitem{PS}  \textit{J. Pal, P. Simon,} On a generalization of the comncept of
derivate, Acta Math. Hung., 29(1977), 155-164.

\bibitem{PSTW}  \textit{L.-E. Persson, F. Schipp, G. Tephnadze and F. Weisz,} An analogy of the Carleson-Hunt theorem with respect to Vilenkin systems, J. Fourier Anal. Appl., (to appear).

\bibitem{pt} \textit{L. E. Persson and G. Tephnadze,} A sharp boundedness result concerning some maximal operators of Vilenkin-Fejér means, Mediterr. J. Math., 13, 4 (2016) 1841-1853.

\bibitem{PTT} \textit{ L. E. Persson, G. Tephnadze and G. Tutberidze,} On the boundedness of subsequences of Vilenkin-Fejér means on the martingale Hardy spaces, operators and matrices, 14, 1 (2020), 283-294.

\bibitem{PTT2} \textit{	L. E. Persson, G. Tephnadze, G. Tutberidze and P. Wall,} Strong summability result of Vilenkin-Fejér means on bounded Vilenkin groups, Ukr. Math. J., 73 (4), (2021), 544-555.

\bibitem{PTW} \textit{L. E. Persson, G. Tephnadze and P. Wall,} On the maximal operators of Vilenkin-Nörlund means, J. Fourier Anal. Appl., 21, 1 (2015), 76-94.

\bibitem{PTW2} \textit{L. E. Persson, G. Tephnadze and P. Wall,} On the Nörlund logarithmic means with respect to Vilenkin system in the martingale Hardy space $H_{1}$, Acta math. Hung., 154, 2 (2018) 289-301. 


\bibitem{PTW3} \textit{L. E. Persson, G. Tephnadze and P. Wall,} On an approximation of 2-dimensional Walsh-Fourier series in the martingale Hardy spaces, Ann. Funct. Anal., 9, 1 (2018), 137-150.

\bibitem{PTWeisz} \textit{L. E. Persson, G. Tephnadze and F. Weisz,} Martingale Hardy Spaces and Summability of Vilenkin-Fourier Series, Springer, (to appear).

\bibitem{s1} \textit{F. Schipp,} Pointwise convergence of expansions with respect to certain product systems, Anal. Math., 2, (1976) 65-76.

\bibitem{s2} \textit{F. Schipp,} Universal contractive projections and a.e. 
convergence, Probability Theory and Applications, Essays to the Memory of
J\'ozsef Mogyor\'odi. Eds.: J. Galambos, I. K\'atai.
Kluwer Academic Publishers, Dordrecht, Boston, London, 1992, 47-75.

\bibitem{s3} \textit{F. Schipp and F. Weisz,} Tree martingales and almost everywhere convergence of Vilenkin-Fourier series, Math. Pannonica, 8 (1997), 17-36.

\bibitem{Sc} \textit{F. Schipp,} Certain rearrangements of series in the Walsh system, 18 (1975), no. 2, 193--201.

\bibitem{sws} \textit{F. Schipp, W.R. Wade, P. Simon and J. Pál,} Walsh
series, An Introduction to Dyadic Harmonic Analysis, Akadémiai Kiadó,
(Budapest-Adam-Hilger (Bristol-New-York)), 1990.

\bibitem{sj1} \textit{P. Sjölin,}  An inequality of Paley and convergence a.e. of 
Walsh-Fourier series,  Arkiv f\"or Math. 8 (1969), 551-570.

\bibitem{steindiv} \textit{E. M. Stein,} On Limits of Sequences of Operators,
Annals of Mathematics, 74, 1 (1961),  140-170.

\bibitem{tep1} \textit{G. Tephnadze,} Fejér means of Vilenkin-Fourier series, Stud. Sci. Math. Hung., 49, 1 (2012) 79-90.

\bibitem{tep2} \textit{G. Tephnadze,} The maximal operators of logarithmic means of one-dimensional Vilenkin-Fourier series, Acta Math. Acad. Paed. Nyíreg., 27 (2011), 245-256.


\bibitem{tep6} \textit{G. Tephnadze,} On the Vilenkin-Fourier coefficients, Georgian Math. J., 20, 1 (2013), 169-177.

\bibitem{tep7} \textit{G. Tephnadze,} On the partial sums of
Vilenkin-Fourier series,  J.
Contemp. Math. Anal. 49(2014),1,23-32.

\bibitem{tep8} \textit{G. Tephnadze,} Strong convergence of two-dimensional Walsh-Fourier series, Ukr. Math. J., 65, 6 (2013), 822-834.

\bibitem{tep9} \textit{G. Tephnadze,} On the partial sums of Walsh-Fourier series, Colloq. Math., 141, 2 (2015), 227-242.

\bibitem{tep10} \textit{G. Tephnadze,} Approximation by Walsh-Kaczmarz-Fejér means on the Hardy space, Acta Math. Sci., 34, 5 (2014), 1593-1602.

\bibitem{tep11} \textit{G. Tephnadze,} On the maximal operators of Walsh-Kaczmarz-Fejér means, Period. Math. Hung., 67, 1 (2013), 33-45.

\bibitem{tep12} \textit{G. Tephnadze,} Strong convergence theorems of Walsh-Fejér means, Acta Math. Hung., 142, 1 (2014), 244-259.

\bibitem{tep13} \textit{G. Tephnadze,} The One-dimensional Martingale Hardy Spaces and Partial Sums and Fej\'er Means with respect to Walsh system, Memoirs on Differential Equations and Mathematical Physics, (to appear).

\bibitem{TT1}	\textit{G. Tephnadze and  G. Tutberidze,} A note on the maximal operators of the Nörlund logaritmic means of Vilenkin-Fourier series, 	Trans. Razmadze Math. Inst., 174, 1 (2020), 1070-112.

\bibitem{Tut1}	\textit{G. Tutberidze,} A note on the strong convergence of partial sums with respect to Vilenkin system, J. Contemp. Math. Anal., 54, 6, (2019), 319-324.


\bibitem{Yano} \textit{S.H. Yano,} Ces\`aro summability of Walsh-Fourier series, T\'ohoku Math. J. 9 (1957), no. 2, 267-272.

\bibitem{wk} \textit{F. Weisz,} Martingale Hardy spaces and their applications in Fourier-analysis. Lecture Notes in Math. vol. 1568, Berlin, Heidelberg, New York: Springer 1994.

\bibitem{We1} \textit{F. Weisz,} Martingale Hardy spaces and their
applications in Fourier Analysis, Springer, Berlin-Heidelberg-New York, 1994.

\bibitem{We3} \textit{F. Weisz,} Hardy spaces and Ces\`{a}ro means of
two-dimensional Fourier series, Bolyai Soc. Math. Studies, (1996), 353-367.

\bibitem{We2} \textit{F. Weisz,} Ces\`aro summability of one- and
two-dimensional Walsh-Fourier series, Anal. Math. 22 (1996), no. 3, 229--242.

\bibitem{Zy}  \textit{A. Zygmund,} Trigonometric Series, Vol. 1, Cambridge Univ.
Press, 1959.



\end{thebibliography}
\end{document}